\documentclass[twoside, 12pt]{article}
\usepackage{amssymb, amsmath, mathrsfs, amsthm}
\usepackage{graphicx}
\usepackage{color}
\usepackage{authblk}

\usepackage[top=1in, bottom=1in, left=1in, right=1in]{geometry}
\usepackage{float, caption, subcaption}
\usepackage[numbers,sort&compress]{natbib}
\usepackage{indentfirst}
\usepackage{setspace}
\newtheorem{thm}{Theorem}[section]
\newtheorem{cor}[thm]{Corollary}
\newtheorem{lem}[thm]{Lemma}

\newtheorem{obs}[thm]{Observation}
\newtheorem{pro}[thm]{Proposition}

\newcommand\equ[2]
{
\begin{equation}\label{#1}
	#2
\end{equation}
}

\usepackage{color}

\newcommand \good[1]
{
V_{\mathrm{good}}(#1)
}

\newcommand \free[1]
{
V_{\mathrm{free}}^{0}(#1)
}
\newcommand \Free[1]
{
V_{\mathrm{free}}^{-}(#1)
}
\newcommand \bad[1]
{
V_{\mathrm{bad}}(#1)
}

\newcommand \fix[1]
{
V_{\mathrm{fix}}(#1)
}
\newcommand \fixk[1]
{
V_{\mathrm{fix}}^{k}(#1)
}

\begin{document}
\title{Vertex Removal and Vertex-Edge Domination\thanks{This work is supported by the National Natural Science Foundation of China (Nos.~12461065 and 12261074).}
}

\author[1\thanks{Email: liminhuimath@163.com.}]
{Minhui Li}

\author[2,3\thanks{Corresponding author. Email: zhangshumin@qhnu.edu.cn.}]
{Shumin Zhang}
\author[4\thanks{Email: fengming.dong@nie.edu.sg and donggraph@163.com.}]
{Fengming Dong}

\affil[1]{\footnotesize Center for General Education, Qinghai Institute of Technology, Xining, Qinghai 810006, China}

\affil[2]{\footnotesize
School of Mathematics and Statistics, Qinghai Normal University, Xining, Qinghai 810001, China}

\affil[3]{\footnotesize
The State Key Laboratory of Tibetan Intelligent Information Processing and Application, Xining, Qinghai 810008, China}

\affil[4]{\footnotesize National Institute of Education, Nanyang Technological University, Singapore}
	
\maketitle
\begin{abstract}
A vertex $v$ of a graph $G$ $ve$-{\it dominates} every edge incident with a vertex in its closed neighborhood $N_{G}[v]$.
A set of vertices is a {\it vertex-edge dominating set} ($ve$-{\it dominating set}) if every edge
of $G$ is $ve$-dominated by some vertex in the set.
The minimum cardinality of such a set is the {\it vertex-edge
domination number} $\gamma_{ve}(G)$, also known as the isolation number.
In this paper, we study the effect of vertex removal on
$\gamma_{ve}(G)$, establish sharp bounds on $\gamma_{ve}(G-v)$ and characterize the vertices whose removal decreases, preserves, or increases
$\gamma_{ve}(G)$.
Let $V^0(G)$ and $V^+(G)$ denote the sets of vertices whose removal
preserves and increases $\gamma_{ve}(G)$, respectively,
and let $V_{\mathrm{fix}}^0(G)$ consist of the vertices in
$V^0(G)$ belonging to every minimum $ve$-dominating set.
We prove that
\[
|V^0(G)|
\geq
3|V_{\mathrm{fix}}^0(G)|
+
\sum_{v\in V^+(G)}
\bigl(\gamma_{ve}(G-v)-\gamma_{ve}(G)+1\bigr).
\]
In particular, $|V^0(G)|\geq 2|V^+(G)|$.
We also prove that every vertex in $V_{\mathrm{fix}}^0(G)$
lies on a cycle and that $V^0(T)\neq\emptyset$ for every
tree $T$ of order at least three.
\end{abstract}


\section{Introduction}
All graphs considered in this paper are finite, simple, and undirected.
For a graph $G$, let $V(G)$ and $E(G)$ denote its vertex set and edge set,
respectively. The \textit{open neighborhood} of a vertex $v$ is
$N_G(v)=\{u\in V(G):uv\in E(G)\}$, and its \textit{closed neighborhood} is
$N_G[v]=N_G(v)\cup\{v\}$. The \textit{degree} of $v$ is $d(v)=|N_G(v)|$.
A vertex $v$ is called an \textit{isolated vertex} if $d(v)=0$.
If $d(v)=1$, then $v$ is called a \textit{leaf}.
For $S\subseteq V(G)$, let
$N_G(S)=\bigcup_{v\in S}N_G(v)$ and $N_G[S]=N_G(S)\cup S$.
Note that $N_G(S)$ may contain vertices of $S$.
We omit the subscript $G$ when the graph is clear from the context.

For any $S\subseteq V(G)$, let $G[S]$ denote the subgraph induced by $S$,
and write $G-S=G[V(G)\setminus S]$. We abbreviate $G-\{v\}$ to $G-v$.
A set $S$ is \textit{independent} if $G[S]$ is edgeless.
For $E'\subseteq E(G)$,
let $G[E']$ denote the edge subgraph with edge set $E'$
and vertex set consisting of the endvertices of edges in $E'$.
The \textit{path}, \textit{cycle}, and \textit{star} of order
$n$ are denoted by $P_n$, $C_n$, and $K_{1,n-1}$, respectively; $2K_2$
denotes the disjoint union of two edges. A \textit{forest} is a graph containing no cycle, and a \textit{tree} is a connected forest.
The \textit{distance} $d(u,v)$ between two vertices $u$ and $v$ in a graph $G$ is the length of a shortest path in $G$ which connects $u$ and $v$,
and $d(u,v)=\infty$ when $u$ and $v$
belong to distinct components.

For any vertex $v$ in a graph $G$,
we say that $v$ \textit{$ve$-dominates}
every edge
in $G$ which is incident with
some vertex in $N[v]$.
A subset $S\subseteq V(G)$ is a \textit{vertex-edge dominating set}
of $G$, abbreviated \textit{$ve$-dominating set}, if every edge $e$ in $G$ is $ve$-dominated by some vertex $v$ in $S$.
The \textit{vertex-edge domination number}, abbreviated \textit{$ve$-domination number}, of $G$ is the minimum cardinality of a $ve$-dominating set of $G$, denoted by $\gamma_{ve}(G)$.
We call a $ve$-dominating set of cardinality
$\gamma_{ve}(G)$ a \emph{$\gamma_{ve}(G)$-set}. Equivalently, $S$ is a
$ve$-dominating set if and only if $G-N[S]$ is edgeless. Thus,
$\gamma_{ve}(G)$ equals the isolation number $\iota(G)$ studied by Caro
and Hansberg~\cite{CH2017}. When $G$ is edgeless or $G$ has no vertices, $\gamma_{ve}(G)=0$. The parameter is
additive over components, since each component must be $ve$-dominated
independently. In particular, removing an isolated vertex does not
change $\gamma_{ve}(G)$.

Vertex-edge domination was introduced by Peters~\cite{Peters1986}
and further studied by Lewis~\cite{Lewis2007} and Lewis et al.~\cite{Lewis2010}.
The same parameter was investigated by Caro and
Hansberg~\cite{CH2017} under the name of the isolation number.
Bounds and structural properties of vertex-edge domination
have been established for general graphs and trees
\cite{Boutrig2016,Klostermeyer2021,Krishnakumari2014,Zylinski2019}.
Algorithmic aspects have also been studied for various graph
classes, including cubic graphs and unit disk graphs
\cite{Ziemann2020,Jena2022,Paul2022}.

The structure of minimum $ve$-dominating sets and the effects of
graph modifications have also been studied.
Senthilkumar, Chellali, Naresh Kumar and Venkatakrishnan~\cite{Senthilkumar2023}
studied graphs with a unique minimum
$ve$-dominating set and gave characterizations of the trees with this
property. More recently, Bartolo, Borg, Dettlaff, Lema\'nska and \.Zyli\'nski~\cite{Bartolo2026} investigated
isolation criticality under multiple edge subdivision. In a related direction,
Thakkar and Jamvecha~\cite[Section~1, p.~245]{Thakkar2017} studied vertex removal under a different stated definition: $v$ dominates $xy$ when $\{x,y\}\subseteq N[v]$. Our definition requires only $\{x,y\}\cap N[v]\ne\emptyset$.
In this paper, we study the effect of vertex removal on
$\gamma_{ve}(G)$ by examining the structure of $\gamma_{ve}(G)$-sets.

Removing a vertex may decrease, preserve, or increase
the $ve$-domination number. Accordingly, we partition $V(G)$
into the following three classes:
$$
\begin{aligned}
V^-(G)&=\{v\in V(G):\gamma_{ve}(G-v)<\gamma_{ve}(G)\},\\
V^0(G)&=\{v\in V(G):\gamma_{ve}(G-v)=\gamma_{ve}(G)\},\\
V^+(G)&=\{v\in V(G):\gamma_{ve}(G-v)>\gamma_{ve}(G)\}.
\end{aligned}
$$
Let $\mathcal{CVR}_{ve}$ denote the class of \textit{vertex-edge domination vertex-critical graphs}, or $\gamma_{ve}$-vertex-critical graphs for short, consisting of all graphs $G$ satisfying $\gamma_{ve}(G)\neq \gamma_{ve}(G-v)$ for every $v\in V(G)$.
Thus $G\in \mathcal{CVR}_{ve}$ if and only if $V^0(G)=\emptyset$.

The vertices of $G$ can also be classified in terms of $\gamma_{ve}(G)$-sets as follows.
Let $\mathcal{S}(G)$ denote the family of all $\gamma_{ve}(G)$-sets,
and let $\mathcal{S}(v,G)=\{S\in\mathcal{S}(G):v\in S\}$.
A vertex $v$ is \textit{$\gamma_{ve}(G)$-good}
if $v\in S$ for some $\gamma_{ve}(G)$-set $S$,
and \textit{$\gamma_{ve}(G)$-bad}
if $v\notin S$ for every $\gamma_{ve}(G)$-set $S$.
It is \textit{$\gamma_{ve}(G)$-fixed} if $v\in S$ for every $\gamma_{ve}(G)$-set $S$, and
\textit{$\gamma_{ve}(G)$-free} if $v$ is contained in some $\gamma_{ve}(G)$-set but not in all $\gamma_{ve}(G)$-sets.
The corresponding sets are denoted by $\good{G}, \bad{G}, \fix{G}, V_{\mathrm{free}}(G)$, respectively.
In particular, the fixed, free, and bad vertices form a partition of $V(G)$.
We further define

$$
\begin{aligned}
\free{G}&=V_{\mathrm{free}}(G)\cap V^0(G),\\
\Free{G}&=V_{\mathrm{free}}(G)\cap V^-(G),\\
\fixk{G}&=\{v\in \fix{G}:
                  \gamma_{ve}(G-v)=\gamma_{ve}(G)+k\}
                  \qquad (k\geq0),
\end{aligned}
$$
where $k$ is an integer. In particular,
$V_{\mathrm{fix}}^{0}(G)=V_{\mathrm{fix}}(G)\cap V^0(G)$.

By associating pairwise disjoint sets of private-edge
endvertices with fixed vertices (Theorem~\ref{thm7}), we obtain
the following weighted bound (Corollary~\ref{V0Vplus}):
$$
|V^0(G)|\geq 3|V_{\mathrm{fix}}^{0}(G)|+
\sum_{v\in V^+(G)}
\bigl(\gamma_{ve}(G-v)-\gamma_{ve}(G)+1\bigr).
$$
We also prove that every vertex in $V_{\mathrm{fix}}^{0}(G)$ lies on
a cycle (Theorem~\ref{lemF}) and that $V^0(T)\neq\emptyset$
for every tree $T$ of order at least three
(Theorem~\ref{thm2}).

Section~2 presents the main results and their proofs.
Section~3 discusses consequences for vertex-removal critical
graphs and questions for further study.
\section{Main results}
In this section, we state
the main results on the effect of vertex removal on the $ve$-domination number. We begin with sharp bounds on the $ve$-domination number after the removal of a vertex.

\begin{thm}\label{thm9}
Let $G$ be a graph and $v\in V(G)$ be a non-isolated vertex.
Then
\equ{eq1}
{
	\gamma_{ve}(G)-1\leq
\gamma_{ve}(G-v)\leq\gamma_{ve}(G)
+d(v)-1.
}
The lower bound also holds for isolated vertices.
Moreover, the bounds are sharp.
\end{thm}

\begin{proof}
Let $S$ be a $\gamma_{ve}(G-v)$-set.
Then $S\cup \{v\}$ is a $ve$-dominating set of $G$.
Hence
$\gamma_{ve}(G)\leq \gamma_{ve}(G-v)+1$, that is, $\gamma_{ve}(G-v)\geq \gamma_{ve}(G)-1$.

Now we prove that $\gamma_{ve}(G-v)\leq\gamma_{ve}(G)+d(v)-1$.
Let $S'$ be a $\gamma_{ve}(G)$-set.
If $v\in S'$, then $(S'\backslash \{v\})\cup N(v)$ is a $ve$-dominating set of $G-v$, since every edge of $G-v$ $ve$-dominated by $v$ is incident with a vertex of $N(v)$.
Therefore, we have $\gamma_{ve}(G-v)\leq\gamma_{ve}(G)+d(v)-1$.
If $v\notin S'$, then $S'$ is also a $ve$-dominating set of $G-v$, so $\gamma_{ve}(G-v)\leq\gamma_{ve}(G)$.
Hence, we have $\gamma_{ve}(G-v)\leq\gamma_{ve}(G)+d(v)-1$,
where the last inequality follows from $d(v)\geq1$.

For sharpness of the lower bound, let $G= K_{1,k}$,
where $k\geq1$, and let $v$ be its center.
Then $\gamma_{\mathrm{ve}}(G)=1$ and
$\gamma_{\mathrm{ve}}(G-v)=0$.
For sharpness of the upper bound,
let $G$ be obtained from $K_{1,r}$~($r\geq1$)
by subdividing each edge once, and let $v$ be the original center.
Then $\{v\}$ $ve$-dominates $G$, while $G-v\cong r K_2$. Hence
$
\gamma_{ve}(G-v)=r=\gamma_{ve}(G)+d(v)-1.
$
\end{proof}

Theorem~\ref{thm9} shows that removing a non-isolated vertex $v$
decreases the $ve$-domination number by at most one, while it may
increase the $ve$-domination number by as much as $d(v)-1$.
Consequently, we have the following.

\begin{cor}\label{cor1}
Let $G$ be a graph and $v\in V(G)$.
If $v\in V^{-}(G)$, then  $\gamma_{ve}(G-v)=\gamma_{ve}(G)-1$.
If $v\in V^{+}(G)$, then  $\gamma_{ve}(G-v)=\gamma_{ve}(G)+r$
for some integer $r$ with $1\le r\le d(v)-1$.
\end{cor}
\subsection{On vertices in $V^-(G)$}
In this subsection, we characterize the vertices in $V^{-}(G)$
in terms of private edges of minimum $ve$-dominating sets.

Recall that, for $S\subseteq V(G)$ and $v\in S$,
an edge $e\in E(G)$ is a \textit{private edge} of $v$
with respect to $S$ if $v$ $ve$-dominates $e$
and no vertex of $S\setminus\{v\}$ does.
The set of all such edges is denoted by $pe_{G}[v,S]$.

\begin{obs}\label{obss1}
Let $G$ be a graph and $v\in V^{-}(G)$.
Then for every $S\in \mathcal{S}(G-v)$ and every $u\in N_{G}[v]$, $S\cup\{u\}\in \mathcal{S}(G)$.
\end{obs}
\begin{proof}
Let $S\in\mathcal S(G-v)$ and $u\in N_G[v]$.
By Corollary~\ref{cor1}, $|S|=\gamma_{ve}(G)-1$.
The set $S$ $ve$-dominates every edge of $G-v$, and $u$
$ve$-dominates every edge incident with $v$.
Thus $S\cup\{u\}$ is a $ve$-dominating set of $G$, and
$$\gamma_{ve}(G)\le |S\cup\{u\}|\le |S|+1=\gamma_{ve}(G).$$
Hence $S\cup\{u\}\in\mathcal S(G)$.
\end{proof}
The following characterization is the analogue of
\cite[Theorem~6]{Thakkar2017}
for the definition used here.
The proof follows the same private-edge argument.
\begin{thm}\label{thm1}
Let $G$ be a graph and $v\in V(G)$. Then $v\in V^{-}(G)$ if and only if there exists a $\gamma_{ve}(G)$-set $S$
with $v\in S$ such that
$\emptyset \ne pe_{G}[v,S]\subseteq\{uv: u\in N_{G}(v)\}$.
\end{thm}

\begin{proof}
Suppose first that $v\in V^{-}(G)$, and let $S'$ be a $\gamma_{ve}(G-v)$-set. By Observation \ref{obss1}, $S=S'\cup \{v\}$ is a $\gamma_{ve}(G)$-set.
If $pe_{G}[v,S]=\emptyset$, then $S'$ is a $ve$-dominating set of $G$,  contradicting the assumption that $v\in V^{-}(G)$.
Thus, $pe_{G}[v,S]\neq\emptyset$.
Now we prove $pe_{G}[v,S]\subseteq\{uv: u\in N_{G}(v)\}$. If there exists some edge $e\in pe_{G}[v,S]$ that is not incident with $v$,
then $e\in E(G-v)$
and it is not $ve$-dominated by any vertex in $S'$.
This contradicts the assumption that $S'$ is a $\gamma_{ve}$-set of $G-v$. Hence, $pe_{G}[v,S]\subseteq\{uv: u\in N_{G}(v)\}$.

Conversely, suppose that there exists a $\gamma_{ve}(G)$-set $S$
with $v\in S$ such that $pe_{G}[v,S]\neq\emptyset$ and $pe_{G}[v,S]\subseteq\{uv: u\in N_{G}(v)\}$.
Then $pe_{G}[v,S]\cap E(G-v)=\emptyset$, implying that $S\backslash\{v\}$ is a $ve$-dominating set of $G-v$.
Thus, $\gamma_{ve}(G-v)\le |S|-1=\gamma_{ve}(G)-1$.
Hence $v\in V^{-}(G)$.
\end{proof}

The following result is a consequence of Theorem \ref{thm1}.

\begin{cor}\label{corcvr}
Let $G$ be a graph. Then $V(G)=V^{-}(G)$ if and only if for every $v\in V(G)$, there exists a $\gamma_{ve}(G)$-set $S$ such that $v\in S$ and $\emptyset \ne pe_{G}[v,S]\subseteq\{uv: u\in N_{G}(v)\}$.
\end{cor}

Next, we present several properties of $v\in V^{-}(G)$.

\begin{pro}\label{pro1}
Let $G$ be a graph. Then the following statements hold:

(1) $V^{-}(G)= \Free{G}$;

(2) For every $v\in V^{-}(G)$, $\bad{G}\cup N_{G}(v)\subseteq \bad{G-v}$;

(3) For every $v\in V^{-}(G)$, $\fix{G}\subseteq \fix{G-v}.$
\end{pro}
\begin{proof}
(1) Since $\Free{G}=V_{\mathrm{free}}(G)\cap V^{-}(G)$, $\Free{G}\subseteq V^{-}(G)$. Conversely,
let $v\in V^{-}(G)$ and $S'$ be a $\gamma_{ve}(G-v)$-set.
Since isolated vertices belong to $V^0(G)$, we have $N_G(v)\neq\emptyset$.
Then by Observation \ref{obss1}, both $S'\cup \{v\}$ and $S'\cup \{u\}$ are $\gamma_{ve}(G)$-sets, where $u\in N_{G}(v)$.
Thus, $v$ is $\gamma_{ve}(G)$-free. Hence $V^{-}(G)\subseteq \Free{G}$.

(2)
For any $u\in N_{G}(v)$,
suppose that $u\notin \bad{G-v}$. Then there exists a $\gamma_{ve}(G-v)$-set $S$ containing $u$. Since $u$ $ve$-dominates every edge incident with $v$, $S$ is also a $ve$-dominating set of $G$. Hence
$\gamma_{ve}(G)\leq |S|=\gamma_{ve}(G)-1$,
a contradiction. Thus $N_{G}(v)\subseteq \bad{G-v}$.

Let $w\in \bad{G}$. If there is a set $S\in \mathcal{S}(w,G-v)$, then $S\cup\{v\}\in \mathcal{S}(w,G)$. This contradicts $w\in \bad{G}$.
Thus, we have $w\in \bad{G-v}$. Hence $\bad{G}\cup N_{G}(v)\subseteq  \bad{G-v}$.

(3) By (1), $v\notin \fix{G}$. Let $v'\in \fix{G}$ and $S'$ be any $\gamma_{ve}(G-v)$-set. By Observation \ref{obss1}, $S'\cup \{v\}$ is a $\gamma_{ve}(G)$-set.
Since $v'\neq v$ and $v'\in \fix{G}$, we have $v'\in S'$. Thus $v'\in \fix{G-v}$, and hence $\fix{G}\subseteq \fix{G-v}$.
\end{proof}
\subsection{On vertices in $V^+(G)$}
In this subsection, we characterize the
vertices $v\in V^+(G)$.
The following observation follows directly from the definition.

\begin{obs}\label{obss2}
Let $G$ be a graph and $S\in \mathcal{S}(G)$. Then for every vertex $v\in V(G)\backslash S$, $S$ is a $ve$-dominating set of $G-v$.
\end{obs}
\begin{pro}\label{prov+}
Let $G$ be a graph and $v\in V(G)$. If $v\in V^{+}(G)$, then $v\in \fix{G}$.
\end{pro}
\begin{proof}
Let $v\in V^{+}(G)$.
Suppose that $v\notin \fix{G}$. Then there exists a $\gamma_{ve}(G)$-set $S$ with $v\notin S$.
By Observation \ref{obss2},
$S$ is also a $ve$-dominating set of $G-v$.
Thus, $\gamma_{ve}(G-v)\leq |S|=\gamma_{ve}(G)$. This contradicts $v\in V^{+}(G)$.
\end{proof}

By Proposition~\ref{prov+}, no free vertex belongs to $V^{+}(G)$.
Thus $V_{\mathrm{free}}(G)
=V_{\mathrm{free}}^{-}(G)\cup V_{\mathrm{free}}^{0}(G).$ 

To characterize the vertices in $V^{+}(G)$, we introduce the following exchange parameter.
Let $S$ be a $\gamma_{ve}(G)$-set containing the vertex $v$. For
$R\subseteq S\setminus\{v\}$, define
\[
\rho_S(v,R)=
\min\left\{
|X|:
X\subseteq V(G)\setminus S
\text{ and }
(S\setminus(R\cup\{v\}))\cup X \text{ $ve$-dominates } G-v
\right\}.
\]
We call $\rho_S(v,R)$ the \textit{exchange number} of $v$
with respect to $S$ and $R$.
Every vertex of $S$ has a neighbor outside $S$;
otherwise, it would have no private edge, contradicting
the minimality of $S$.
Thus $V(G)\setminus S$ dominates $G-v$, so
$X=V(G)\setminus S$ is feasible for every
$R\subseteq S\setminus\{v\}$. Consequently,
\[
  0\leq \rho_S(v,R)\leq |V(G)\setminus S|.
\]
\begin{lem}\label{lem:exchange}
Let $G$ be a graph, $v\in V(G)$, and $S$ be a
$\gamma_{ve}(G)$-set containing $v$. Then
\[
\gamma_{ve}(G-v)
=
\gamma_{ve}(G)-1+
\min_{R\subseteq S\setminus\{v\}}
\bigl(\rho_S(v,R)-|R|\bigr).
\]
\end{lem}

\begin{proof}
For any
$R\subseteq S\setminus\{v\}$, let $X\subseteq V(G)\setminus S$
be a set of cardinality $\rho_S(v,R)$ such that
$D=(S\setminus(R\cup\{v\}))\cup X$
is a $ve$-dominating set of $G-v$. Since $X\cap S=\emptyset$, we have
\[
|D|
=\gamma_{ve}(G)-1-|R|+\rho_S(v,R).
\]
Therefore,
\equ{eq4}
{
\gamma_{ve}(G-v)
\leq
\gamma_{ve}(G)-1+
\min_{R\subseteq S\setminus\{v\}}
\bigl(\rho_S(v,R)-|R|\bigr).
}
Conversely, let $D$ be a $\gamma_{ve}(G-v)$-set, and let
\[
R=(S\setminus\{v\})\setminus D
\quad\text{and}\quad
X=D\setminus S.
\]
Then
$R\subseteq S\setminus\{v\},$
$X\subseteq V(G)\setminus S,$
and
$D=(S\setminus(R\cup\{v\}))\cup X.$
Hence, by the definition of $\rho_S(v,R)$, $\rho_S(v,R)\leq |X|.$

Moreover,
\[
|D|
=\gamma_{ve}(G)-1-|R|+|X|
\geq
\gamma_{ve}(G)-1+\rho_S(v,R)-|R|.
\]
It follows that
\equ{eq5}
{
\gamma_{ve}(G-v)
\geq
\gamma_{ve}(G)-1+
\min_{R'\subseteq S\setminus\{v\}}
\bigl(\rho_S(v,R')-|R'|\bigr).
}

By~(\ref{eq4}) and~(\ref{eq5}), we have \[
\gamma_{ve}(G-v)
=
\gamma_{ve}(G)-1+
\min_{R\subseteq S\setminus\{v\}}
\bigl(\rho_S(v,R)-|R|\bigr).
\]
\end{proof}

\begin{thm}\label{thm3}
Let $G$ be a graph and $v\in V(G)$.
Then $v\in V^{+}(G)$ if and only if there exists a
$\gamma_{ve}(G)$-set $S$ containing $v$ such that
$\rho_S(v,R)\geq |R|+2$
for every $R\subseteq S\setminus\{v\}$.
\end{thm}
\begin{proof}
Suppose that $v\in V^{+}(G)$, and let $S$ be a $\gamma_{ve}(G)$-set.
By Proposition~\ref{prov+}, $v\in S$.
By Lemma~\ref{lem:exchange}, we have
\[
\min_{R\subseteq S\setminus\{v\}}
\bigl(\rho_S(v,R)-|R|\bigr)
=\gamma_{ve}(G-v)-\gamma_{ve}(G)+1
\geq2.
\]
Hence $\rho_S(v,R)\geq |R|+2$ for every
$R\subseteq S\setminus\{v\}$.

Conversely, suppose that such a set $S$ exists.
By Lemma~\ref{lem:exchange},
\[
\gamma_{ve}(G-v)
=\gamma_{ve}(G)-1+
\min_{R\subseteq S\setminus\{v\}}
\bigl(\rho_S(v,R)-|R|\bigr)
\geq\gamma_{ve}(G)+1.
\]
Thus $v\in V^{+}(G)$.
\end{proof}

To further investigate the structure of vertices in $V^{+}(G)$, we first
establish the following lemma concerning the vertex-edge domination number.
\begin{lem}\label{lem:2K2C5}
Let $G$ be a graph. If $\gamma_{ve}(G)\geq 2$, then $G$ contains
either an induced copy of $2K_2$ or a copy of $C_5$.
\end{lem}

\begin{proof}
Suppose that $\gamma_{ve}(G)\geq2$
and that $G$ contains neither an induced copy of $2K_2$
nor a copy of $C_5$.
Let $v$ be a vertex of maximum degree in $G$.
We claim that $\{v\}$ is a $ve$-dominating set of $G$.

Suppose not. Then $G-N[v]$ contains an edge $xy$. For every
$u\in N(v)$, the edges $vu$ and $xy$ do not induce a $2K_2$.
Since $v$ is adjacent to neither $x$ nor $y$, the vertex $u$ must
be adjacent to at least one of $x$ and $y$.

Let
\[
A=N(v)\cap N(x)
\quad\text{and}\quad
B=N(v)\cap N(y).
\]
Then
$$
A\cup B=N(v).
$$
If there exist $u\in A\setminus B$ and $z\in B\setminus A$, then
$v-u-x-y-z-v$ forms a $C_5$, a contradiction. Hence either
$A\subseteq B$ or $B\subseteq A$.

Without loss of generality, suppose that $A\subseteq B$. Then
$B=N(v)$. Since $xy\in E(G)$ and $x\notin N[v]$, we have
\[
d(y)\geq |N(v)|+1=d(v)+1,
\]
contradicting the choice of $v$ as a vertex of maximum degree. Therefore, $\{v\}$ is a $ve$-dominating set of $G$, and hence
$\gamma_{ve}(G)=1$, a contradiction.
\end{proof}

The converse of Lemma~\ref{lem:2K2C5} does not hold in general.
For example, take an induced $2K_2$ and add a
vertex adjacent to all four of its vertices. Then this new vertex
$ve$-dominates every edge of the resulting graph.

We now apply Lemma~\ref{lem:2K2C5} to the private edges of a
vertex in $V^{+}(G)$.

\begin{thm}\label{thm4}
Let $G$ be a graph and $v\in V^{+}(G)$. For any
$\gamma_{ve}(G)$-set $S$, let $U=V(G[pe_{G}[v,S]])\setminus \{v\}$.
Then the vertex-induced subgraph $G[U]$ contains either an induced copy of $2K_2$ or a copy of $C_5$.
\end{thm}

\begin{proof}
Since $v\in V^{+}(G)$, by Proposition~\ref{prov+}, we have $v\in \fix{G}$,
and hence $v\in S$.
We claim that $\gamma_{ve}(G[U])\geq 2.$

If $E(G[U])=\emptyset$, then every private edge of $v$ with respect
to $S$ is incident with $v$. Thus, $S\setminus\{v\}$ is a
$ve$-dominating set of $G-v$, contradicting $v\in V^{+}(G)$.
Hence $E(G[U])\neq\emptyset$.

Suppose that $\gamma_{ve}(G[U])=1$, and let $x\in U$ be a
$ve$-dominating vertex of $G[U]$. Every edge in
$E(G-v)\setminus pe_{G}[v,S]$ is $ve$-dominated by some vertex of
$S\setminus\{v\}$, while every edge in
$pe_{G}[v,S]\cap E(G-v)$ belongs to $G[U]$ and is $ve$-dominated
by $x$. Therefore,
$(S\setminus\{v\})\cup\{x\}$
is a $ve$-dominating set of $G-v$. Consequently,
\[
\gamma_{ve}(G-v)
\leq |(S\setminus\{v\})\cup\{x\}|
\leq |S|
=\gamma_{ve}(G),
\]
contradicting $v\in V^{+}(G)$.

Thus $\gamma_{ve}(G[U])\geq 2$. By
Lemma~\ref{lem:2K2C5}, $G[U]$ contains either an induced
$2K_2$ or a $C_5$.
\end{proof}

The next theorem shows that no vertex of $V^{+}(G)$
is adjacent to a vertex of $V^{-}(G)$.
\begin{thm}\label{thm6}
For two vertices $u\in V^{+}(G)$ and $v\in V^{-}(G)$, $d(u,v)\geq2$. Moreover, the bound is sharp.
\end{thm}
\begin{figure}
        \centering
        \includegraphics[width=10cm]{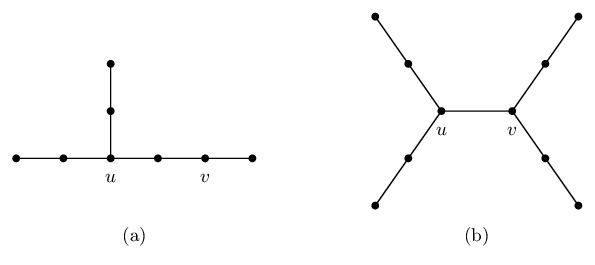}
        \caption{(a) $u\in V^{+}(G)$ and $v\in V^{-}(G)$ with $d(u,v)=2$.
(b) $u,v\in V^{+}(G)$.}
        \label{fig}
\end{figure}
\begin{proof}Clearly, $u$ and $v$ are not isolated vertices.
Suppose that there exist two vertices $u\in V^{+}(G)$ and $v\in V^{-}(G)$ such that $d(u,v)<2$. Then $uv\in E(G)$.
Let $S_{v}$ be a $\gamma_{ve}(G-v)$-set with cardinality $\gamma_{ve}(G)-1$.
We must have $u \in S_{v}$; otherwise, $S_{v}\cup\{v\}$ is a $\gamma_{ve}(G)$-set and $u \notin S_{v}\cup\{v\}$. This contradicts $u\in V^{+}(G)$.
Since $uv\in E(G)$, $u$ $ve$-dominates all edges which are incident with $v$. Thus, $S_{v}$ is a $ve$-dominating set of $G$.
This contradicts $v\in V^{-}(G)$. Hence, we have $d(u,v)\geq2$.

To see the sharpness of the bound, consider the graph in Figure \ref{fig}(a), 
where $\gamma_{ve}(G)=2$,
$\gamma_{ve}(G-u)=3$, and $\gamma_{ve}(G-v)=1$.
Thus $u\in V^+(G)$, $v\in V_{\mathrm{free}}^-(G)$,
and $d(u,v)=2$.
\end{proof}

The following is an immediate consequence of Theorem~\ref{thm6}.

\begin{cor}\label{cor7}
Let $G$ be a graph and $v\in V^{+}(G)$. Then
$N(v)\subseteq V^{+}(G)\cup V^{0}(G).$
\end{cor}
The conclusion of Corollary~\ref{cor7} cannot be strengthened to
$N(v)\subseteq V^{0}(G)$.
Indeed, in the graph $G$ in Figure~\ref{fig}(b), $u$ and $v$ are
adjacent and $\gamma_{ve}(G)=2$, whereas
$\gamma_{ve}(G-u)=\gamma_{ve}(G-v)=3$.
\subsection{On vertices in $V^{0}(G)$}
We now study vertices whose removal preserves the $ve$-domination number.
\begin{obs}\label{obsv0}
Let $G$ be a graph. Then $V^{0}(G)=V_{\mathrm{fix}}^{0}(G)\cup \free{G}\cup \bad{G}$.
\end{obs}
\begin{proof}
By Proposition \ref{pro1}(1) and Proposition \ref{prov+},
$V^-(G)\subseteq V_{\mathrm{free}}(G)$ and
$V^+(G)\subseteq V_{\mathrm{fix}}(G)$.
Hence $V_{\mathrm{bad}}(G)\subseteq V^0(G)$.
Intersecting the partition
$V(G)=V_{\mathrm{fix}}(G)\cup V_{\mathrm{free}}(G)
\cup V_{\mathrm{bad}}(G)$
with $V^{0}(G)$ gives the desired equality.
\end{proof}
The graph in Figure~\ref{fig:fixed0} shows that a fixed vertex need
not belong to $V^+(G)$.
Indeed, $\{v\}$ is the unique $\gamma_{ve}(G)$-set, while
$G-v\cong P_5$ and $\gamma_{ve}(G-v)=\gamma_{ve}(G)=1$.
Hence $v\in V_{\mathrm{fix}}^{0}(G)$.

\begin{figure}
        \centering
        \includegraphics[width=4cm]{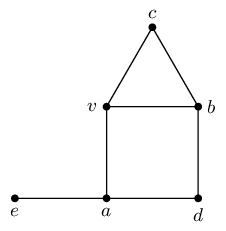}
        \caption{A graph $G$ with $v\in V_{\mathrm{fix}}^0(G)$.}
        \label{fig:fixed0}
\end{figure}

\begin{thm}\label{thm30}
Let $G$ be a graph and $v\in V(G)$. Then $v\in V^{0}(G)$
if and only if one of the following holds:

(1) $v\in \bad{G}$;

(2) there exists a $\gamma_{ve}(G)$-set $S$ containing $v$
    such that $\rho_S(v,R)\geq |R|+1$
    for every $R\subseteq S\setminus\{v\}$, with equality for
    at least one such $R$.
\end{thm}
\begin{proof}
Suppose that $v\in V^{0}(G)$. If
$v\in V_{\mathrm{bad}}(G)$, then (1) holds. Otherwise, there
exists a $\gamma_{ve}(G)$-set $S$ containing $v$. By
Lemma~\ref{lem:exchange}, we have
\[
\min_{R\subseteq S\setminus\{v\}}
\bigl(\rho_S(v,R)-|R|\bigr)
=
\gamma_{ve}(G-v)-\gamma_{ve}(G)+1
=1.
\]
Hence, $\rho_S(v,R)\geq |R|+1$
for every $R\subseteq S\setminus\{v\}$, with equality for at
least one such $R$.

Conversely, by Observation~\ref{obsv0}, we have $\bad{G}\subseteq V^{0}(G)$.
Now suppose that (2) holds. Then
\[
\min_{R\subseteq S\setminus\{v\}}
\bigl(\rho_S(v,R)-|R|\bigr)=1.
\]
By Lemma~\ref{lem:exchange},
$\gamma_{ve}(G-v)=\gamma_{ve}(G).$
Thus $v\in V^{0}(G)$.
\end{proof}

\begin{pro}\label{lems1}
Let $G$ be a graph and $v\in V(G)$.

(1) If $v\in V^{0}(G)$, then $\bigcup_{k\geq 2}V_{\mathrm{fix}}^{k}(G)\subseteq \fix{G-v}$.

(2) If $v\in \free{G}\cup \bad{G}$, then $\fix{G}\subseteq \good{G-v}$.

(3) If $v\in \bad{G}$, then $V^{-}(G)\subseteq V^{-}(G-v)$ and, for every $u\in N_{G}(v)$, $u \notin V^{-}(G-v)$.

\end{pro}
\begin{proof}
(1) Let $u\in \fixk{G}$ for some $k\geq2$,
and $S$ be any $\gamma_{ve}(G-v)$-set. Since $v\in V^{0}(G)$,
$|S|=\gamma_{ve}(G-v)=\gamma_{ve}(G).$
Suppose that $u\notin S$. Then $S\cup\{v\}$ is a $ve$-dominating
set of $G-u$. Hence,
$$
\gamma_{ve}(G-u)
\leq |S\cup\{v\}|
\leq \gamma_{ve}(G)+1.
$$
On the other hand, since $u\in \fixk{G}$,
$$
\gamma_{ve}(G-u)=\gamma_{ve}(G)+k
\geq \gamma_{ve}(G)+2,
$$
a contradiction. Thus $u\in S$.
Since $S$ is arbitrary,
$u\in \fix{G-v}$.

(2) Let $v\in \free{G}\cup \bad{G}$.
Then there exists a $\gamma_{ve}(G)$-set $S$ such that $v\notin S$.
By Observation~\ref{obss2}, $S$ is a $ve$-dominating set of $G-v$.
Moreover,
$|S|=\gamma_{ve}(G)=\gamma_{ve}(G-v),$
and hence $S$ is a $\gamma_{ve}(G-v)$-set. Since every vertex in
$\fix{G}$ belongs to $S$, it follows that
$\fix{G}\subseteq V_{\mathrm{good}}(G-v).$

(3) Let $v'\in V^{-}(G)$ and $S$ be a
$\gamma_{ve}(G-v')$-set. 
Since $v\in \bad{G}$, by Proposition \ref{pro1}(2), we have $v\in \bad{G-v'}$, so $v\notin S$. Thus, $S$ is a $ve$-dominating set of $G-\{v',v\}$.
Therefore, $$\gamma_{ve}(G-\{v',v\})\leq|S|=\gamma_{ve}(G)-1=\gamma_{ve}(G-v)-1.$$
Thus $v'\in V^{-}(G-v)$.
Hence, $V^{-}(G)\subseteq V^{-}(G-v)$.

Now suppose that $u\in N_G(v)\cap V^-(G-v)$, and let
$S'$ be a $\gamma_{ve}(G-\{u,v\})$-set.
Since $uv\in E(G)$, the vertex $v$ $ve$-dominates every edge
incident with $u$ or $v$, so $S'\cup\{v\}$ $ve$-dominates $G$.
Moreover, since $v\in \bad{G}\subseteq V^0(G)$,
\[
|S'\cup\{v\}|=\gamma_{ve}(G-\{u,v\})+1
=\gamma_{ve}(G-v)=\gamma_{ve}(G).
\]
Thus $S'\cup\{v\}$ is a $\gamma_{ve}(G)$-set,
contradicting $v\in \bad{G}$.
Hence, $u\notin V^-(G-v)$ for every $u\in N_G(v)$.
\end{proof}
The following theorem extends the private-edge bound in
\cite[Proposition~2]{Senthilkumar2023} to fixed vertices
without assuming that the minimum $ve$-dominating set is unique.
It also yields the counting bound in Corollary~\ref{V0Vplus}.
\begin{thm}\label{thm7}
Let $G$ be a graph. For each $v\in \fix{G}$, let $S_v$ be a
$\gamma_{ve}(G)$-set and
\[
X_v=
\left\{
x\in N_G(v):
x \text{ is an endvertex of some }
e\in pe_{G}[v,S_v]\cap E(G-v)
\right\}.
\]
Then $X_v\subseteq V^{0}(G)\setminus V_{\mathrm{fix}}^{0}(G)$ and
$|X_v|\geq \max\{2,\gamma_{ve}(G-v)-\gamma_{ve}(G)+1\}$. Moreover, the sets $X_v$ are pairwise disjoint.
\end{thm}

\begin{proof}
We first show that $|X_v|\geq \max\{2,\gamma_{ve}(G-v)-\gamma_{ve}(G)+1\}$. Since $v\in \fix{G}$, $v\in S_v$. If $X_v=\emptyset$, then every
private edge of $v$ with respect to $S_v$ is incident with $v$.
Hence, $S_v\setminus\{v\}$ is a $ve$-dominating set of $G-v$.
Thus $v\in V^{-}(G)$, and by Proposition \ref{pro1}(1),
$v\in V_{\mathrm{free}}(G)$, which contradicts $v\in \fix{G}$.
If $|X_v|= 1$, write $X_v=\{x\}$. Then every edge in
$pe_{G}[v,S_v]\cap E(G-v)$ is incident with $x$. Moreover, since
$x\in N_G(v)$, the vertex $x$ also $ve$-dominates every private edge
of $v$ that is incident with $v$. Thus, $x$ $ve$-dominates every edge
in $pe_{G}[v,S_v]$. 
Every edge outside $pe_G[v,S_v]$ is $ve$-dominated
by some vertex of $S_v\setminus\{v\}$.
Therefore, $(S_v\setminus\{v\})\cup\{x\}$
is a $ve$-dominating set of $G$.
Since $x\neq v$ is an endvertex of a private edge
of $v$ with respect to $S_v$, we have $x\notin S_v$.
Thus $(S_v\setminus\{v\})\cup\{x\}$ is a
$\gamma_{ve}(G)$-set not containing $v$, contradicting $v\in \fix{G}$.
Therefore, $|X_v|\geq 2.$

Moreover, $X_v$ $ve$-dominates every private edge
of $v$ that remains in $G-v$.
Thus, $(S_v\setminus\{v\})\cup X_v$ is a $ve$-dominating set of $G-v$, and
$\gamma_{ve}(G-v)\leq\gamma_{ve}(G)-1+|X_v|.$
It follows that $|X_v|\geq\gamma_{ve}(G-v)-\gamma_{ve}(G)+1$.

We next prove that $X_v\subseteq V^{0}(G)\setminus V_{\mathrm{fix}}^{0}(G)$.
Let $x\in X_v$.
If $x\in V^{-}(G)$, then $v\in N_G(x)$, and
Proposition~\ref{pro1}(2) and (3) gives
$v\in V_{\mathrm{bad}}(G-x)\cap V_{\mathrm{fix}}(G-x),$
a contradiction.
Hence $x\notin V^{-}(G)$.
Suppose that $x\in \fix{G}$.
Then $x\in S_v$.
By the definition of $X_v$, there exists an edge
$e\in pe_{G}[v,S_v]\cap E(G-v)$
incident with $x$. Thus, $x$ $ve$-dominates $e$, contradicting the
fact that $e$ is a private edge of $v$ with respect to $S_v$.
Hence $x\notin \fix{G}$, and therefore $x\notin V^+(G)$.
Thus $X_v\subseteq V^{0}(G)\setminus V_{\mathrm{fix}}^{0}(G)$.

Finally, let $u,v\in\fix{G}$ be distinct and suppose that
$x\in X_u\cap X_v$.
Since $x\in X_v$, there exists an edge
$e\in pe_G[v,S_v]\cap E(G-v)$ incident with $x$.
Since $u\in\fix{G}$, we have $u\in S_v\setminus\{v\}$
for any choice of $S_v$.
Moreover, $x\in X_u\subseteq N_G(u)$, so $u$ $ve$-dominates $e$,
contradicting $e\in pe_G[v,S_v]$.
Hence $X_u\cap X_v=\emptyset$, and the sets $X_v$ are
pairwise disjoint.
\end{proof}

\begin{cor}\label{V0Vplus}
For every graph $G$, 
\begin{equation*}
 |V^0(G)|\geq3|V_{\mathrm{fix}}^{0}(G)|+
\sum_{v\in V^+(G)}\bigl(\gamma_{ve}(G-v)-\gamma_{ve}(G)+1\bigr).
\end{equation*}
In particular,
\begin{equation*}
 |V^0(G)|\geq3|V_{\mathrm{fix}}^{0}(G)|+2|V^+(G)|
 \quad\text{and}\quad
 |\fix{G}|\leq\frac{|V(G)|-|V^-(G)|}{3}.
\end{equation*}
\end{cor}

\begin{proof}By Theorem~\ref{thm7}, the sets $X_v$, $v\in\fix{G}$,
are pairwise disjoint and satisfy
\[
 \bigcup_{v\in\fix{G}}X_v
 \subseteq V^0(G)\setminus V_{\mathrm{fix}}^0(G).
\]
Moreover, $|X_v|\geq 2$ for every $v\in V_{\mathrm{fix}}^0(G)$,
while
$|X_v|\geq \gamma_{ve}(G-v)-\gamma_{ve}(G)+1$
for every $v\in V^+(G)$.
Since
$\fix{G}=V_{\mathrm{fix}}^0(G)\cup V^+(G)$,
we obtain
\begin{align*}
 |V^0(G)|-|V_{\mathrm{fix}}^0(G)|
 &\geq \sum_{v\in\fix{G}}|X_v|\\
 &=\sum_{v\in V_{\mathrm{fix}}^0(G)}|X_v|
   +\sum_{v\in V^+(G)}|X_v|\\
 &\geq 2|V_{\mathrm{fix}}^0(G)|
   +\sum_{v\in V^+(G)}
      \bigl(\gamma_{ve}(G-v)-\gamma_{ve}(G)+1\bigr).
\end{align*}
This proves~$|V^0(G)|\geq3|V_{\mathrm{fix}}^{0}(G)|+
\sum_{v\in V^+(G)}\bigl(\gamma_{ve}(G-v)-\gamma_{ve}(G)+1\bigr).$

Since $\gamma_{ve}(G-v)-\gamma_{ve}(G)\geq1$ for every $v\in V^+(G)$, we have $$|V^0(G)|\geq3|V_{\mathrm{fix}}^{0}(G)|+2|V^+(G)|.$$
Since $V_{\mathrm{fix}}(G)$ is the disjoint union of
$V_{\mathrm{fix}}^0(G)$ and $V^+(G)$, we obtain
$$
\begin{aligned}
|V(G)|-|V^-(G)|
&=|V^0(G)|+|V^+(G)|\\
&\ge 3|V_{\mathrm{fix}}^0(G)|+3|V^+(G)|\\
&=3|V_{\mathrm{fix}}(G)|.
\end{aligned}
$$
\end{proof}

Corollary~\ref{V0Vplus} yields the following characterization of $\mathcal{CVR}_{ve}$.
\begin{cor}\label{thm8}
Let $G$ be a graph. Then $\gamma_{ve}(G)\neq \gamma_{ve}(G-v)$ for every $v\in V(G)$ if and only if $V(G)=V^{-}(G)$.
\end{cor}
\begin{proof}
If $V(G)=V^{-}(G)$, then $\gamma_{ve}(G)\neq \gamma_{ve}(G-v)$ for every $v\in V(G)$. Suppose that $V^{0}(G)=\emptyset$. Then by Corollary \ref{V0Vplus}, we have $0=|V^{0}(G)|\geq 2|V^{+}(G)|$.
So $V^{+}(G)=\emptyset$. Since $V(G)=V^{-}(G)\cup V^{+}(G)\cup V^{0}(G)$,
$V(G)=V^{-}(G)$.
\end{proof}

It follows from Corollary \ref{corcvr} and Corollary \ref{thm8} that $G\in \mathcal{CVR}_{ve}$ if and only if for all $v\in V(G)$, there exists a $\gamma_{ve}(G)$-set $S$ such that $v\in S$ and $pe_{G}[v,S]\subseteq\{uv: u\in N_{G}(v)\}$, where $pe_{G}[v,S]\neq \emptyset$.

Corollary \ref{thm8} also implies that $V^{0}(G)=\emptyset$
only if $V^{+}(G)=\emptyset$. The converse does not hold.
For example, if $G=K_{1,n}$ with $n\geq 2$, then
$V^{+}(G)=\emptyset$, whereas $V^{0}(G)\neq \emptyset$.
We next show that $V^{0}(T)\neq\emptyset$ for every tree $T$
of order at least three.

\begin{thm}\label{thm2}
For any tree $T$ on $n\geq3$ vertices, $V^{0}(T)\neq\emptyset$.
\end{thm}
\begin{proof}Let $v$ be the neighbor of an endvertex of a longest path of $T$.
Then $d(v)\ge2$, and $v$ is adjacent to at least one leaf
and at most one non-leaf vertex.
If $v$ has no non-leaf neighbor, then $T$ is a star. In this case, removing any leaf does not change the $ve$-domination number, and hence
$V^{0}(T)\neq\emptyset$.
Thus, we may assume that $N(v)=\{u_{1},\ldots,u_{k},w\}$, where $k\geq 1$, the vertices $u_1,\ldots,u_k$ are leaves, and $w$ is a non-leaf vertex.

Suppose that $k\geq2$. Fix $i\in\{1,\ldots,k\}$ and let
$S$ be a $\gamma_{ve}(T-u_i)$-set. Since there exists
$j\neq i$, the edge $vu_j$ must be $ve$-dominated by some vertex of
$S$, and hence $S\cap N[v]\neq\emptyset$. Therefore, $S$ also
$ve$-dominates the edge $vu_i$, so $S$ is a $ve$-dominating set of
$T$. Thus,
$\gamma_{ve}(T)\leq \gamma_{ve}(T-u_i).$
Since $d(u_i)=1$, Theorem~\ref{thm9} gives
$\gamma_{ve}(T-u_i)\leq \gamma_{ve}(T).$
Hence $\gamma_{ve}(T-u_i)=\gamma_{ve}(T),$
and therefore $u_i\in V^{0}(T)$.

Now suppose that $k=1$. Then $v$ is adjacent to exactly
one leaf $u_1$ and $d(v)=2$.
Since $d(u_{1})=1$, $\gamma_{ve}(T-u_{1})\leq\gamma_{ve}(T)$. In addition, by Theorem~\ref{thm9}, $\gamma_{ve}(T)\leq\gamma_{ve}(T-u_{1})+1$.
Hence, we have $\gamma_{ve}(T-u_{1})=\gamma_{ve}(T)$ or $\gamma_{ve}(T-u_{1})=\gamma_{ve}(T)-1$.
In the first case, $u_1\in V^0(T)$. 
In the second case, $u_{1}\in V^{-}(T)$.
By Theorem \ref{thm6}, we have $v\notin V^{+}(T)$.
If $v\in V^{0}(T)$, the conclusion follows. 
It remains to consider $v\in V^{-}(T)$.

Let $H=T-\{w,v,u_{1}\}$.
Since $T-w$ is the disjoint union of $H$ and the edge $vu_1$, 
$\gamma_{ve}(T-w)=\gamma_{ve}(H)+1$.
A $\gamma_{ve}(H)$-set together with $w$ $ve$-dominates $T$, so
\[
\gamma_{ve}(T)\leq\gamma_{ve}(H)+1
=\gamma_{ve}(T-w).
\]
Thus $w\notin V^-(T)$. 
Moreover, since $vw\in E(T)$ and
$v\in V^-(T)$, Theorem~\ref{thm6} gives $w\notin V^+(T)$.
Hence $w\in V^0(T)$.
\end{proof}

The next result relates the vertices in $V_{\mathrm{fix}}^{0}(G)$
to cycles in $G$.

\begin{thm}\label{lemF}
Let $G$ be a graph and $v\in V(G)$. If $v\in V_{\mathrm{fix}}^{0}(G)$, then there exists a cycle $C$ in $G$ such that $v\in V(C)$.
\end{thm}
\begin{proof}
Let $v\in V_{\mathrm{fix}}^{0}(G)$. Since the $ve$-domination number is
additive over components, we may assume that $G$ is connected.
Suppose that there is no cycle containing $v$ in $G$.

Since $v\in V_{\mathrm{fix}}^{0}(G)$, $v$ is not isolated.
Let $N(v)=\{u_1,\ldots,u_r\}.$ Then distinct vertices in $N(v)$ belong to distinct components of $G-v$; otherwise, there exists a cycle containing $v$.
Let these components be $H_1,\ldots,H_r$, where $u_i\in V(H_i)$.
Then each $u_i$ is the only vertex of $H_i$ adjacent to $v$.
For each $i$, let
$$
a_i=\min\big\{|A|: A\subseteq V(H_i)
\text{ and }A\cup\{v\} \ \
ve\text{-dominates }G[V(H_i)\cup\{v\}]\big\}.
$$
Set $b_i = \gamma_{ve}(H_i)$, and choose $B_i \in \mathcal{S}(H_i)$. Since
$B_i\cup\{v\}$ $ve$-dominates $G[V(H_i)\cup\{v\}]$, we have
$a_i\leq |B_i|=b_i.$
Moreover, since $u_i$ is the only vertex of $H_i$
adjacent to $v$, adding $u_i$ suffices to $ve$-dominate all edges
of $H_i$ that may be $ve$-dominated by $v$.
Thus $b_i\leq a_i+1.$
Hence
\equ{eq8}
{b_i-a_i\in\{0,1\}.}

Let $S$ be a $\gamma_{ve}(G)$-set.
Since $v$ is fixed, $v\in S$.
For each $i$, no vertex of another component of $G-v$
can $ve$-dominate an edge of $H_i$, so
$|S\cap V(H_i)|\geq a_i$.
Conversely, if $A_i$ attains $a_i$ for each $i$,
then $\{v\}\cup\bigcup_{i=1}^r A_i$ $ve$-dominates $G$.
Hence
\equ{eq9}
{\gamma_{ve}(G)=1+\sum_{i=1}^{r}a_i.}

On the other hand,
\equ{eq10}
{
\gamma_{ve}(G-v)=\sum_{i=1}^{r}b_i.
}
Since $v\in V_{\mathrm{fix}}^{0}(G)$,
$\gamma_{ve}(G-v)=\gamma_{ve}(G).$

By~(\ref{eq8})--(\ref{eq10}), we have
$\sum_{i=1}^{r}(b_i-a_i)=1.$
Hence, there is a unique index $j$ such that
$b_j=a_j+1,$
while $b_i=a_i$ for every $i\neq j$.
Choose $A_j$ attaining $a_j$, and for each $i\neq j$, choose $B_i\in\mathcal{S}(H_i)$. Then $A_j\cup\{u_j\}$ is a $ve$-dominating set of $H_j$. We have $u_j\notin A_j$; otherwise, $A_j$ would $ve$-dominate
$H_j$, contradicting $b_j=a_j+1$.
Since $|A_j\cup\{u_j\}|=a_j+1=b_j,$
the set $A_j\cup\{u_j\}$ is a $\gamma_{ve}(H_j)$-set.
Let $D=\left(\bigcup_{i\neq j}B_i\right)\cup A_j\cup\{u_j\}.$
Then $D$ $ve$-dominates every edge of $G-v$. Moreover, since
$u_j\in N(v)$, the vertex $u_j$ $ve$-dominates every edge
incident with $v$. Therefore, $D$ is also a $ve$-dominating set of $G$.
Furthermore,
$$|D|=\sum_{i=1}^{r}b_i=\gamma_{ve}(G-v)=\gamma_{ve}(G),$$
so $D$ is a $\gamma_{ve}(G)$-set not containing $v$, contradicting
$v\in \fix{G}$.
\end{proof}
\begin{cor}\label{thmF}
For every forest $F$, $\fix{F}=V^+(F)$.
\end{cor}
\begin{proof}
By Proposition~\ref{pro1}(1) and Proposition~\ref{prov+},
$\fix{F}=V_{\mathrm{fix}}^0(F)\cup V^+(F)$.
Since $F$ contains no cycle, Theorem~\ref{lemF} gives
$V_{\mathrm{fix}}^0(F)=\emptyset$, and the result follows.
\end{proof}
This extends the implication (i)$\Rightarrow$(iii) of
\cite[Theorem~7]{Senthilkumar2023} to fixed vertices
in arbitrary forests, without assuming that the minimum
$ve$-dominating set is unique.
\section{Conclusion}
In this paper, we studied the effect of vertex removal
on the vertex-edge domination number.
We established sharp bounds on $\gamma_{ve}(G-v)$
and characterized the three vertex classes
$V^-(G)$, $V^0(G)$, and $V^+(G)$.
We also introduced an exchange parameter,
derived a weighted counting bound,
and proved that every fixed vertex in $V^0(G)$
lies on a cycle.

Our results imply that a nonempty forest $F$ belongs to
$\mathcal{CVR}_{ve}$ if and only if each component is isomorphic
to $K_2$. Indeed, every isolated vertex belongs to $V^0(F)$.
By Theorem~\ref{thm2} and additivity, every component
of order at least three also contains a vertex of $V^0(F)$.
Conversely, if all components are copies of $K_2$, deleting
any vertex decreases $\gamma_{ve}(F)$ by one.

For paths and cycles, the standard formulas
(see Peters~\cite{Peters1986} and
Lewis~\cite[Proposition~23]{Lewis2007}) are
\[
\gamma_{ve}(P_m)=\left\lceil\frac{m-1}{4}\right\rceil
\quad(m\geq1),
\qquad
\gamma_{ve}(C_n)=\left\lceil\frac{n}{4}\right\rceil
\quad(n\geq3).
\]
Since $C_n-v\cong P_{n-1}$ for every $v\in V(C_n)$,
we have
$C_n\in\mathcal{CVR}_{ve}$ if and only if
$n\equiv1,2\pmod4$.

It would be interesting to obtain further structural
characterizations of graphs in $\mathcal{CVR}_{ve}$,
to characterize graphs satisfying $V^0(G)=V(G)$,
and to determine whether the coefficients in the weighted
bound of Corollary~\ref{V0Vplus} can be improved.

\end{document}